\documentclass[11pt]{article}
\usepackage{euscript}
\usepackage{amsmath}
\usepackage{amssymb}
\usepackage{graphicx}
\usepackage{color}
\definecolor{dgreen}{rgb}{0,.6,0}

\usepackage{theorem}
\usepackage{amsmath}
\usepackage{amsfonts}
\usepackage{amssymb}
\usepackage{latexsym}
\usepackage{rotating}
\newtheorem{theorem}{\bf Theorem}[section]

\newtheorem{coro}[theorem]{\bf Corollary}

\newtheorem{defn}[theorem]{\bf Definition}
\newtheorem{remark}[theorem]{\bf Remark}
\newenvironment{proof}{\noindent{\em Proof:}}{\quad \hfill$\Box$\vspace{2ex}}

\def\TT{{\mathbb T}}

\def\ZZ{{\mathbb Z}}
\def\NN{{\mathbb N}}
\def\RR{{\mathbb R}}

\def\RRd{{\mathbb R}^d}

\def\ZZd{{\mathbb Z}^d}
\def\ZZdp{{\mathbb Z}^d_+}
\def\TTd{{\mathbb T}^d}

\def\k2{K^\varphi_2(\TT)}

\def \bK {\Bbb K}

\def \bB {\Bbb B}

\def \bT {\Bbb T}

\def \and {\, \mbox{\rm and}\, }

\makeatletter

\newcommand{\Rmnum}[1]{\expandafter\@slowromancap\romannumeral #1@}
\makeatother

\def\TT{{\mathbb T}}

\def\ZZ{{\mathbb Z}}
\def\NN{{\mathbb N}}
\def\RR{{\mathbb R}}

\def\RRd{{\mathbb R}^d}

\def\ZZd{{\mathbb Z}^d}
\def\ZZdp{{\mathbb Z}^d_+}
\def\TTd{{\mathbb T}^d}

\newlength{\fixboxwidth}
\begin{document}
\title{\sffamily Approximation by translates of a single function 
of functions in space induced by the convolution with a given function}
\author{Dinh D\~ung$^a$\footnote{Corresponding author. Email: dinhzung@gmail.com.},
Charles A. Micchelli$^b$ and Vu Nhat Huy$^c$ \\\\
$^a$ Vietnam National University, Information Technology Institute \\
144 Xuan Thuy, Hanoi, Vietnam  \\\\
$^b$Department of Mathematics and Statistics, SUNY Albany \\
Albany, 12222, USA \\\\
$^c$ College of Science,  Vietnam National University\\
334 Nguyen Trai, Thanh Xuan, Ha Noi\\\\
}
\date{\ttfamily  September 26, 2016 --  Version 2.0}
 \tolerance 2500
\maketitle

\begin{abstract}
We study approximation by arbitrary linear combinations of $n$ translates of a single function 
of periodic functions. We construct some methods of
this approximation for functions in a class induced by the convolution with a given function,
and prove upper bounds of  $L_p$-the approximation convergence rate by these methods, 
when $n \to \infty$, for $1 < p < \infty$, and lower bounds of the quantity of best approximation of this class
by arbitrary linear combinations of $n$ translates of arbitrary function, for the particular case $p=2$. 


\end{abstract}


\section{Introduction}

The purpose of this  paper is to improve and extend the ideas in the
recent papers \cite{1, DC-Err16} on approximation by translates of the
multivariate Korobov function. The motivation for the results given
in \cite{1, DC-Err16}, and those presented here come from Machine Learning,
since certain cases of our results here  relate to approximation of
a function by sections of  a reproducing kernel corresponding to
specific Hilbert space of functions. This relationship to Machine
Learning is described in the papers \cite{1,MXZ06} and is not
reviewed in detail here. Nonetheless, in this regard, we recall that
the observation presented in \cite{MXZ06} provide necessary and
sufficient conditions for sections of a reproducing kernel to be
dense in continuous functions in the corresponding Hilbert space.
This result begs the question of the convergence rate of approximation by
sections of a reproducing kernel.
We refer the reader to \cite{1, DC-Err16} for detailed survey and bibliography
on the problems considered in the present paper. 
Here, in this paper, we introduce
a weighted Hilbert space of multivariate periodic functions and
provide insights into this question. The results presented here also
extend other norms on multivariate periodic functions and these
results are presented separately in this paper. 

We shall begin the study of this problem with a description of the
notation used throughout the paper. In this regard, we merely follow
closely the presentation in \cite{1, DC-Err16}. The $d$-dimensional torus
denoted by $\bT^d$ is the cross product of $d$ copies of the
interval $[0,2\pi]$ with the identification of the end points. When
$d=1$, we simply denote the $d$-torus by $\bT$. Functions on $\bT^d$
are identified with functions on $\RRd$ which are $2\pi$ periodic in
each variable. We shall denote by $L_p(\TTd), \ 1 \le p < \infty$,
the space of integrable functions on $\TTd$ equipped with the norm
\begin{equation}\nonumber
\|f\|_p \ := (2\pi)^{-d/p}\left(\int_{\TTd} |f(\bold{x})|^p
d\bold{x}\right)^{1/p}
\end{equation}
and we shall only consider only real valued functions on $\TTd$. However, all
the results in this paper are true in the complex setting. Also, we
will use the Fourier series of a real valued function in complex
form.

For vectors ${\mathbf x}:=(x_l:l=1,2,\ldots,d)$ and
$\bold{y}:=(y_l:l=1,2,\ldots,d)$ in $\bT^d$ we use
$(\bold{x},\bold{y}):=\sum_{l=1}^{d}x_ly_l$ for the inner product of
$\bold{x}$ with $\bold{y}$.
Given any integrable function $f$ on $\TTd$ and any
lattice vector $\bold{j}=(j_l: l=1,2,\ldots,d) \in \ZZd$, we let
${\hat f}(\bold{j})$ denote the $\bold{j}$-th Fourier coefficient of
$f$ defined by the equation
\[
{\hat f}(\bold{j}) \ := \ (2\pi)^{-d}\int_{\TTd} f(\bold{x}) \,
e^{i(\bold{j},{\mathbf x}) }\, d\bold{x}.
\]
 Frequently, we use the superscript notation $\bB^d$ to  denote
the cross product of $d$ copies of  a given set $\bB$ in $\mathbb{R}^d$.

Let $1 \le p \le \infty$ and $p'$ be defined be the equation $1/p + 1/p' = 1$. 
Assume that $\varphi_{\lambda, d}$ belongs to $L_{p'}(\TTd)$ and can be represented as the Fourier series
\begin{equation} \label{varphi}
\varphi_{\lambda, d} \ = \ \sum_{\bold{j} \in \ZZd}
\lambda_{\mathbf j}^{-1} \, e^{i({\mathbf j},\cdot)}
\end{equation}
in distributional sense for some sequence
$\lambda:= (\lambda_{\bf j}: {\bf j}\in \mathbb{Z}^d)$  with nonzero
components. Notice that if $(\lambda_{\bf j}^{-1}: {\bf j}\in \mathbb{Z}^d)$  is absolutely summable, 
the function $\varphi_{\lambda,d}$ is continuous on $\bT^d$.
In the case that $d=1$ we merely  write $\varphi_\lambda$ for the
univariate function $\varphi_{\lambda,1}$. The special case
when $d=1$, $r \in (0, \infty)$ and $\lambda$ is given by the
equation
\begin{equation} \nonumber
\lambda_j =
\begin{cases}
|j|^r &\text{if }  j \ne 0, \\
1 &\text{if } j=0.
\end{cases}
\end{equation}
In this case, the function $\varphi_{\lambda}$ corresponds to the   Korobov
function which was the focus of study in \cite{1}. In general, we
introduce a subspace of $L_p(\TTd)$ defined as
\[\Phi_{\lambda, p}(\bT^d):=\{f: f =
\varphi_{\lambda,d}*g, \ g \in L_p(\TTd)\}
\]
 with norm
\[
\|f\|_{\Phi_{\lambda, p}(\TT^d)} \ := \|g\|_p
\]
where we denote the convolution of any two functions $f_1$ and $f_2$
on $\bT^d$, as $f_1*f_2$, and as usual, define it at
$\bold{x}\in\bT^d$ by equation
\[
(f_1*f_2)(\bold{x}) \ := (2\pi)^{-d}\int_{\TTd} f_1(\bold{x}) \,
f_2(\bold{x} - \bold{y}) \, d\bold{y},
\]
whenever the integrand is in $L_1(\bT^d)$.

The space $\Phi_{\lambda, 2}(\TT^d)$ is particularly  interesting as it has an
interpretation  in Machine Learning which is described in detail in
the papers \cite{1,MXZ06}. As in the paper \cite{1} we are concerned
with the following concept. Let $\mathbb{W} \subset L_p(\TTd)$ be a
prescribed subset of $L_p(\TTd)$ and $\psi \in L_p(\TTd)$ be a given
function on $\TTd$. Set $\NN_n:= \{1,2,...,n\}$. We are interested in the approximation in
$L_p(\TTd)$-norm of all functions $f \in \mathbb{W}$ by arbitrary
linear combinations of $n$ translates of the function $\psi$, that
is,  by the  functions  in the set $\{ \psi(\cdot - {\mathbf y}_l):
\ {\mathbf y}_l \in \TTd, l\in \mathbb{N}_n \}$ and measure the
error in terms of the quantity
\begin{equation}  \nonumber
M_n(\mathbb{W},\psi)_p
 := \
\sup_{f \in \mathbb{W}}
\ \inf_{c_l\in\mathbb{R},{\mathbf y}_l\in\TTd}
 \biggl\|f - \sum_{l\in \mathbb{N}_n} c_l \psi(\cdot - {\mathbf y}_l)\biggl\|_p.
\end{equation}
The aim of the present paper is to investigate the convergence rate,
when $n\rightarrow\infty$, of  $M_n(U_{\lambda, p}(\bT^d), \psi)_p$,
where $U_{\lambda, p}(\bT^d)$ is the unit ball in $\Phi_{\lambda, p}(\bT^d)$. We shall
also obtain a lower bound for the convergence rate as
$n\rightarrow\infty$ of the quantity
\begin{equation}  \nonumber
M_n(U_{\lambda, 2}(\bT^d))_2  
:= \
\inf_{\psi \in L_2(\TTd)}  M_n(U_{\lambda, 2}(\bT^d),\psi)_2
\end{equation}
which 
gives information about the best choice of $\psi$.

This paper is organized in the following manner. In Section two we
introduce the method of approximation used throughout the paper and
provide error estimates for both the univariate and multivariate
cases. In Section three we apply these results to the problem
described earlier, in particular, of approximating periodic
functions by sections of reproducing kernels. We continue this line
of investigation in Section four by relying upon observations of V.
Maiorov  \cite{Ma05}  as a means to establish lower bounds of approximation.

\section{A linear method of univariate approximation} 
\label{Univariate approximations}

In this section, we  introduce a method of approximation induced by
translates of the function  defined  in equation (\ref{varphi}) in
the univariate case.  We do this in some greater generality.  To the end, we start with the functions
$\varphi_\lambda, \varphi_\beta$ of the form given in equation (\ref{varphi}).
We introduce a trigonometric polynomial  $H_m$  defined at $x \in \TT$  as
\begin{equation}\label{Hm}
H_m(x) =\sum_{|k| \le m}\frac{\beta_k}{\lambda_k} \,e^{ik x}:= \sum_{|k| \le m} \alpha_k \,e^{ik x},
\end{equation}
that is, $\alpha_k:= \frac{\beta_k}{\lambda_k}$ for $|k| \le m$.  
For a function $f \in \Phi_{\lambda, p} (\TT)$ represented as $f= \varphi_\lambda * g$, $g \in L_p(\TT)$, we define the operator
\begin{equation}\label{Qm}
Q_{m}(f):= \ \frac{1}{2m+1}\sum_{l=0}^{2m}V_m(g) (\delta_m l)\varphi_\beta(\cdot - \delta_m l),
\end{equation}
where $\delta_m :=  2\pi/(2m+1)$ and
$
V_m(g):=H_m *g.
$
Our goal is to obtain an estimate for the error of approximating a function  $f \in \Phi_{\lambda, p} (\TT)$ by $Q_{m}(f)$ a linear combination of $2m + 1$ translates of the function $\varphi_\beta$. 
For the moment, we assume that $\varphi_\beta \in L_p(\TT)$ and put 
\[I_{m,j}=\{k: \, k\in \mathbb{Z}, \ (2m+1)j - m \leq k\leq (2m+1)j +m \}.
\]

\subsection{Error estimates for functions in the space $\Phi_{\lambda, 2} (\mathbb{T})$}

\begin{theorem} \label{theorem[error0]}
We put
\begin{equation*}
\varepsilon_m= \max \{ \sup_{|k| > m}|\lambda_k^{-1}| , \sqrt{ \sum_{|k| \in \NN} \Gamma_{m,k}^2}\},
\end{equation*}
where $\gamma_k=\alpha_{k^{'}} \beta_k^{-1}$, $k'$ the unique
integer in $[-m,m]$ such that the number $(k-k')/(2m+1)$ is an
integer, and 
$ \Gamma_{m,j}=\max\{|\gamma_k|: \ k\in I_{m,j}\}$.
Then there exists a positive constant $c$ such that for
all
 $m\in \mathbb{N}$ and $f\in \Phi_{\lambda, 2} (\mathbb{T})$ we have that
\begin{equation}\nonumber
\|f-Q_{m}(f)\|_2 
\ \leq \
c\,\varepsilon_m \|f\|_{\Phi_{\lambda, 2} (\mathbb{T})},
\end{equation} 
and consequently,
\begin{equation}  \nonumber
M_{2m+1}(U_{\lambda, 2}(\bT),\varphi_\beta)_2
\ \leq \
c\,\varepsilon_m. 
\end{equation}
\end{theorem}

\begin{proof}
We define { the kernel $P_m(x,t)$}  for $x,t \in \mathbb{T}$ as
\begin{equation*}
P_m (x,t) := \ \frac{1}{2m+1} \sum\limits_{l=0}^{2m}\varphi_\beta
(x-\delta_m l) H_m (\delta_m l -t)
\end{equation*}
and easily  obtain from   our definition (\ref{Qm})  the  equation
\begin{equation} \nonumber
Q_{m}(f) (x)=\int_{\mathbb{T}} P_m (x,t) g(t)\, dt.
\end{equation}
We now use  equation (\ref{varphi}), the definition of the
trigonometric polynomial $H_m$ given in equation  (\ref{Hm}) and the
easily verified fact,   for $k,s \in \ZZ, s \in [-m,m]$, that
\begin{equation*}
\frac{1}{2m+1}\sum_{l=0}^{2m}  e^{i k(t-\delta_m l)}  e^{is(\delta_m
l -t)}=
\begin{cases}
0,  \quad& \text { if } \ \frac{k-s}{2m+1} \not\in \mathbb{Z}\\[2ex]
e^{i(k-k^{'})t },   \quad & \text{ if } \ \frac{k-s}{2m+1} \in \mathbb{Z}
\end{cases}
\end{equation*}
 to conclude that
\begin{equation} \nonumber
P_m (x,t)= \ \sum_{k\in \mathbb{Z}}  \gamma_k {e^{i ( k-k^{'}) x} }    .
\end{equation}
We again use the formula for the function $\varphi_\lambda$ given in equation (\ref{varphi}) to get that
\begin{equation} \nonumber
P_m(x,t) -\varphi_\lambda(x-t) = \sum_{k\in \mathbb{Z}}   e^{i kx}  (  \gamma_k e^{-ik^{'}t}-\lambda_k^{-1} e^{-ikt}  ).
\end{equation}
For  $|k| \le m$ we have that   $k=k^{'} \leq 2m$ and the above expression becomes
\begin{equation} \nonumber
P_m(x,t) -\varphi_\lambda(x-t) = 
\sum_{|k|> m}  \gamma_k e^{i kx}   e^{-ik^{'}t}- \sum_{|k|> m}\lambda_k^{-1} e^{-ikt}.
\end{equation}
From this equation  and the definition of $f$ we deduce the formula
\begin{equation}\nonumber
Q_{m}(f)(x) - f(x) 
\ = \ 
\sum_{|k|> m}  \gamma_k \hat{g}({k^{'}}) e^{i kx} - \sum_{|k|> m}\lambda_k^{-1} \hat{g}(k) e^{i kx}
 =: A_m(x) - B_m(x).
\end{equation}
By the triangle inequality we have
\begin{equation}\label{[|Q_m(f) -f|_2<]}
\|Q_{m}(f) -f\|_2
\ \le \
\|A_m\|_2 + \|B_m\|_2.
\end{equation}
Parseval's identity  gives us the equation
\begin{equation} \nonumber
\begin{split}
\|A_m\|_2^2
&= \sum_{|k| > m} |\gamma_k|^2  |\hat{g}({k^{'}})|^2 \\
&= 
\sum_{|j| \in \NN } \ \sum_{k\in I_{m,j}} |\gamma_k|^2  |\hat{g}({k- (2m+1)j})|^2 \\
&\leq  
\sum_{|j| \in \NN} \ \sum_{k\in I_{m,j}} |\Gamma_{m,j}|^2  |\hat{g}({k- (2m+1)j})|^2.
\end{split}
\end{equation}
Hence, by our assumption on $\varepsilon_m$  and the inequality
\[
\sum_{k\in I_{m,j}}  |\hat{g}({k- (2m+1)j})|^2 \leq \|g\|_2^2,
\]
 we obtain that
 \begin{equation} \label{eeq5}
\|A_m\|_2^2
\ \leq \  \sum_{|j| \in \NN} \Gamma_{m,j}^2 \|g\|_2^2
\ \leq \ \varepsilon_m^2 \|g\|_2^2.
\end{equation}
Next, by using Parseval's identity again, we have that
\begin{equation*}
\|B_m\|_2^2
=
\sum_{|k| > m}   |\lambda_k^{-2}| |\hat{g}({k})|^2 \leq \sup_{|k|>m}|\lambda_k^{-1}| \sum_{|k| > m}
|\hat{g}({k})|^2 \leq  \varepsilon_m^2 \|g\|_2^2.
\end{equation*}
This, together with \eqref{[|Q_m(f) -f|_2<]} and \eqref{eeq5}, proves the theorem.
\end{proof}

\begin{defn}
The sequence$\{\theta_k: k \in \ZZ\}$ will be called a
nondecreasing-type sequence  if there exists a positive constant $c$ such that $\theta_k \geq c \theta_l$ for all $k,l\in \ZZ$ satisfying the inequality $|k| > |l|$.
\end{defn}

\begin{theorem} \label{Error[d=1,p=2]}
Let  $\Big\{\frac{|\beta_k|}{|\alpha_k|} :k\in \ZZ \Big\}$ and  
$\{|\lambda_k|:k\in \ZZ \}$ be nondecreasing-type sequences. Then there
exists a positive constant $c$ such that  for all
 $m\in \mathbb{N}$ and $f\in \Phi_{\lambda, 2} (\mathbb{T})$ we have that
\begin{equation*}
\|f-Q_{m}(f)\|_2 
\leq 
c\,\|f\|_{\Phi_{\lambda, 2}(\mathbb{T})}\sqrt{ \sum_{|k|\in \mathbb{N}} \lambda_{mk}^{-2}},
\end{equation*}
and consequently,
\begin{equation}  \nonumber
M_{2m+1}(U_{\lambda, 2}(\bT),\varphi_\beta)_2
\ \leq \
c\,\sqrt{ \sum_{|k|\in \mathbb{N}} \lambda_{mk}^{-2}}. 
\end{equation}
\end{theorem}

\begin{proof}
From our hypothesis we have that $|\gamma_k| \le c_1 |\lambda_k^{-1}|$,  
\[
\sup_{|k| > m}|\lambda_k^{-1}|  
\leq  
c_2 \, |\lambda_m^{-1}|
\leq
c_2\,\sqrt{ \sum_{|k| \in \mathbb{N}} \lambda_{mk}^{-2}}
\]
 and 
$\Gamma_{m,k} \leq c_3|\lambda_{mk}^{-1}|$ 
for all $k \in \mathbb{N}$ with some positive constants $c_1$, $c_2$ and $c_3$. 
From these inequalities and Theorem~\ref{theorem[error0]}, the proof of the result is complete.
\end{proof}

From this theorem we have the following result.
\begin{coro} \label{coro2.5}
Let $|\lambda_k|=|\beta_k|=|\lambda_{-k}|=|\beta_{-k}|$ for all $k\in \ZZ$,
and $\Big\{\frac{|\lambda_k|}{k^r}:k\in \ZZ \Big\}$ be a nondecreasing-type sequence for
some $r> \frac{1}{2}$. Then there exists a positive constant $c$
such that  for all
 $m\in \mathbb{N}$ and $f\in \Phi_{\lambda, 2} (\mathbb{T})$ we have that
\begin{equation*}
\|f-Q_{m}(f)\|_2 \ \leq \ c\, |\lambda_m^{-1}| \|f\|_{\Phi_{\lambda, 2} (\mathbb{T})},
\end{equation*}
and consequently,
\begin{equation}  \nonumber
M_{2m+1}(U_{\lambda, 2}(\bT),\varphi_\beta)_2
\ \leq \
c\,|\lambda_m^{-1}|. 
\end{equation}
\end{coro}

\begin{proof}
We see from the hypothesis that
$$\frac{|\lambda_{mk}|}{(mk)^r} \geq c'\,\frac{|\lambda_m|}{m^r}$$
for some positive constant $c'$ and then from which it follows that
\[
\frac{|\lambda_{m k}|}{|\lambda_m|}\geq c'\, k^r
\]
for all $m,k\in \NN$.
Hence, we conclude that
\[
c'\,\sqrt{ \sum_{k\in \mathbb{N}} \lambda_{mk}^{-2}}  \leq  |\lambda_m^{-1}|  \sqrt{\sum_{k\in \mathbb{N}}k^{-2r}} \ .
\]
Note that, since $r>\frac{1}{2}$ we have  $\sum_{k\in \mathbb{N}}k^{-2r} < \infty$ and then by applying 
Theorem  \ref{Error[d=1,p=2]}
we complete the proof.
\end{proof}

\begin{coro} \label{coro2.6}
Let $|\beta_k|=|\beta_{-k}|=|\lambda_k|^2=|\lambda_{-k}|^2$ for all $k\in \ZZ$, and 
$\{|\lambda_k |:k\in\ZZ\}$ be  a nondecreasing-type sequence.
Then there exists a positive constant $c$ such that  for all
 $m\in \mathbb{N}$ and $f\in \Phi_{\lambda, 2} (\mathbb{T})$ we have that
\begin{equation*}
\|f-Q_{m}(f)\|_2 \leq c\, \|f\|_{\Phi_{\lambda, 2} (\mathbb{T})}
\sqrt{ \sum_{k\in \mathbb{N}} \lambda_{mk}^{-2}}, 
\end{equation*}
and consequently,
\begin{equation}  \nonumber
M_{2m+1}(U_{\lambda, 2}(\bT),\varphi_\beta)_2
\ \leq \
c\,\sqrt{ \sum_{|k|\in \mathbb{N}} \lambda_{mk}^{-2}}. 
\end{equation}
\end{coro}

\begin{proof}
The inclusions $\varphi_\lambda, \varphi_\beta \in L_2(\TT)$ and the equations $|\beta_k|=|\beta_{-k}|=|\lambda_k|^2=|\lambda_{-k}|^2$
for all $k\in \ZZ$ yield that  
$|\beta_k^{-1}| \le c' \, |\lambda_k^{-1}|$ for all $k\in \ZZ$ with a positive constants $c'$.
Hence, by applying Theorem  \ref{Error[d=1,p=2]} we prove the corollary.
\end{proof}

Similarly to the proof of Corollary \ref{coro2.5} we can prove the following fact.

\begin{coro} \label{coro2.7}
Let $|\beta_k|=|\beta_{-k}|=|\lambda_k|^2=|\lambda_{-k}|^2$ for all $k\in \ZZ$ and 
$\Big\{\frac{|\lambda_k|}{k^r}:k\in\ZZ \Big\}$ be a nondecreasing-type sequence $r> \frac{1}{2}$. 
Then there exists a positive constant $c$ such that for all
 $m\in \mathbb{N}$ and $f\in \Phi_{\lambda, 2} (\mathbb{T})$ we have that
\begin{equation}\nonumber
\|f-Q_{m}(f)\|_2 \ \leq \ c\, |\lambda_m^{-1}| \|f\|_{\Phi_{\lambda, 2} (\mathbb{T})},
\end{equation}
and consequently,
\begin{equation}  \nonumber
M_{2m+1}(U_{\lambda, 2}(\bT),\varphi_\beta)_2
\ \leq \
c\, |\lambda_m^{-1}|. 
\end{equation}
\end{coro}

\begin{coro} \label{coro2.8}
Let $r > \frac{1}{2}$ and  $\lambda_k=\beta_k=|k|^r$ for all $k\in \ZZ, k\ne 0$, and $\lambda_0 = 1$. 
Then there exists a positive constant $c$ such that  for all   $m\in \mathbb{N}$ and 
$f\in \Phi_{\lambda, 2} (\mathbb{T})$ we have that
\begin{equation*}
\|f-Q_{m}(f)\|_2 \ \leq \ c\, m^{-r} \|f\|_{\Phi_{\lambda, 2} (\mathbb{T})},
\end{equation*}
and consequently,
\begin{equation}  \nonumber
M_{2m+1}(U_{\lambda, 2}(\bT),\varphi_\beta)_2
\ \leq \
c\, m^{-r}. 
\end{equation}
\end{coro}

\begin{remark} Note that under the assumptions of Corollary \ref{coro2.6}, $K_\lambda(x,y):= \varphi_\beta(x-y)$ is the reproducing kernel for the Hilbert space $\Phi_{\lambda, 2} (\mathbb{T})$.  This
means, for every function $f\in \Phi_{\lambda, 2} (\mathbb{T})$ and
$x\in\bT$, we have that
\begin{equation} \nonumber
f(x) \ = (f,K_\lambda(\cdot,x))_{\Phi_{\lambda, 2} (\mathbb{T})},
\end{equation}
where $(\cdot,\cdot)_{\Phi_{\lambda, 2} (\mathbb{T})}$ denotes the inner product on
the Hilbert space $\Phi_{\lambda, 2} (\mathbb{T})$. 
It is known that the linear span of the set of functions 
$\bK_\lambda:= \{K_\lambda(\cdot - y): y \in \TT\}$ is dense in the Hilbert space $\Phi_{\lambda, 2} (\mathbb{T})$.
Under a certain restriction on the sequence $\lambda$, Corollaries \ref{coro2.6}, \ref{coro2.7} and \ref{coro2.8} give an explicit rate of the error of the linear approximation of 
$f \in \Phi_{\lambda, 2} (\mathbb{T})$ by the function $Q_m(f)$ belonging $\bK_\lambda$.
For a definitive treatment of reproducing kernels, see, for example, \cite{Ar50}. 
Corollary \ref{coro2.8} has been proven  as Theorem 2.10 in \cite{1} where 
$\Phi_{\lambda, 2} (\mathbb{T})$ is the Korobov space $K^r_2(\TT)$.
\end{remark}

\subsection{Error estimates for functions in the space $\Phi_{\lambda, p} (\mathbb{T})$}

For this purpose, we define, for $m \in \mathbb{N} $, the quantity
\begin{equation} \label{lambda}
\varepsilon_m 
:= \ 
\max \{ \sum_{|k| >m} |\Delta \lambda_k^{-1} |,
\sum_{|k| >m} |\Delta \gamma_k| + \sum_{k\in \mathbb{Z}}
|\gamma_{k(2m+1)+m}| \} ,
\end{equation}
where $\gamma_k$ is defined as in Theorem \ref{theorem[error0]}, and 
$\Delta \theta_k:= \theta_k - \theta_{k+1}$ for the sequence $\{\theta_k:\, k\in\NN\}$.

Now, we are ready to state the the following result.
\begin{theorem}\label{2.1}
If $1 < p < \infty$ then there exists a positive constant $c$ such
that for all $f\in \Phi_{\lambda, p} (\TT)$ and $m\in \mathbb{N}$,
we have that
\begin{equation} \label{ineq[|f-Q_m(f)|_p]}
\|f-Q_{m}(f)\|_p \ \leq \ c\, \varepsilon_m \|f\|_{\Phi_{\lambda, p} (\TT)},
\end{equation}
and consequently,
\begin{equation}  \nonumber
M_{2m+1}(U_{\lambda,p}(\bT),\varphi_\beta)_p
\ \leq \
c\,\varepsilon_m. 
\end{equation}
\end{theorem}

\begin{proof}
In order to prove \eqref{ineq[|f-Q_m(f)|_p]} we need an auxiliary result which is a direct corollary of the well-known Marcinkiewicz multiplier theorem, see, e.g., \cite[Lemma 2.7]{1}. 
For $r, s \in \ZZ$ and an integrable function $g$ on $\TT$ we introduce 
the trigonometric polynomial
\begin{equation} \nonumber
F_{r,s}(g,x):=\sum_{r \le k \le s} \hat{g}({k})  e^{i kx}.
\end{equation}
Then there exists an absolute positive constant $c'$ such that for all $g \in L_p(\bT)$  and 
$r, s \in \ZZ$ we have that
\begin{equation}\label{Grs}
\|F_{r, s} (g)\|_p \leq c' \|g\|_p.
\end{equation}

In a completely similar way as in the proof of Theorem~\ref{theorem[error0]}, we can establish the formula   
\begin{equation}\label{Qm-f}
Q_{m}(f)(x) - f(x) 
\ = \ 
\sum_{|k|>m}  \gamma_k \hat{g}({k^{'}}) e^{i kx} - \sum_{|k|>m}\lambda_k^{-1} \hat{g}(k) e^{i kx}
\end{equation}
for a function $f \in \Phi_{\lambda, p} (\TT)$ represented as $f= \varphi_\lambda * g$, $g \in L_p(\TT)$.

The next step is to  decompose each sum above into two parts.  Specifically, we write the first sum above as
\begin{equation}\label{A+}
\sum_{k>m}  \gamma_k\hat{g}({k^{'}})  e^{i kx} + \sum_{k<-m}  \gamma_k  \hat{g}({k^{'}}) e^{i kx}.
\end{equation}
We call the   the first   sum in  equation  (\ref{A+}) $A_m^+(x)$ and the other $A_m^-(x)$. Now, we  readily  rewrite $A_m^+(x)$  in the form
\begin{equation} \label{A+1}
A_m^+(x)=
\sum_{j \in \mathbb{N} }\sum_{k\in I_{m,j}} \gamma_k e^{i kx} \hat{g}(k^{'}).
\end{equation}

We shall express  the right hand side of equation  \eqref{A+1} in
an alternate form by using summation by parts.
For this purpose, we introduce the modified difference operator defined on vectors as
\begin{equation} \nonumber
\Lambda\gamma_k=
\begin{cases}
\gamma_k- \gamma_{k+1}, &\text {if } j(2m+1)-m \leq k < j(2m+1)+m \\
\gamma_k, &\text {if }  k= j(2m+1)+m .
\end{cases}
\end{equation}
With this notation in hand and  the fact that, for $k\in I_{m,j}$ we have that $k^{'}= k - j(2m+1)$, we conclude that
\begin{equation} \nonumber
\begin{split}
A_m^+(x) 
&=
\sum_{j \in \mathbb{N} }\sum_{k\in I_{m,j}} \gamma_k e^{i kx} \hat{g}(k - j(2m+1)) \\[1ex]
 & = \sum_{j \in \mathbb{N} }  \sum_{k\in I_{m,j}} \Lambda \gamma_k \, e^{i j(2m+1)x} G_{-m, k - j(2m+1)}.
\end{split}
\end{equation}
Consequently,  according to \eqref{Grs} and the H\"older inequality, there exists a positive constant $c_1$ such that
\begin{equation}\nonumber
\|A_m^+\|_p
 \le
c_1\,\|g\|_p  \sum_{j \in \mathbb{N} }\sum_{k\in I_{m,j}} |\Lambda\gamma_k|.
\end{equation}
From this inequality and the definition of $\varepsilon_m$, given in equation (\ref{lambda}),
we conclude that
\begin{equation} \nonumber
\|A_m^+\|_p
\leq c_1 \varepsilon_m\, \|g\|_p.
\end{equation}
A bound on   $\|A_m^-\|_p$ follows by a similar argument and yields the inequality
\begin{equation} \nonumber
\|A_m^-\|_p
\leq c_1 \varepsilon_m\, \|g\|_p.
\end{equation}

There still remains the task of bounding the second sum in equation (\ref{Qm-f}). As before, we split it into two parts
\begin{equation}\nonumber
\sum_{k>m}  \lambda^{-1}_k e^{i kx} \hat{g}({k}) + \sum_{k<-m}  \lambda^{-1}_k e^{i kx} \hat{g}({k})
\end{equation}
and call the first sum above $B_m^+(x)$  and the second sum  $B_m^-(x)$. As before, summation by parts yields the alternate form
\begin{equation} \nonumber
B_m^+(x)
= \sum_{k \ge m+1} \Delta \lambda_k^{-1} G_{m+1,k }
\end{equation}
 from which  we deduce that
\begin{equation*}
\|B_m^+(x) \|_p \leq
 \sum_{k \ge m+1} |\Delta \lambda_k^{-1}| \| g\|_p.
\end{equation*}
Therefore, by \eqref{lambda} we obtain that
\begin{equation}\nonumber
\|B_m^+(x) \|_p \leq c_1  \varepsilon_m \| g\|_p
\end{equation}
and, in a similar way, we prove that
\begin{equation}\nonumber
 \|B_m^-\|_p
\ \leq c_1 \varepsilon_m\, \|g\|_p.
\end{equation}
Combining our remarks above proves the result.
\end{proof}

Now, we are ready to state the the following result.
\begin{theorem} \label{theorem[error]}
Let  $1 < p < \infty$ and  $\beta_k=\beta_{-k}$, $\lambda_k=\lambda_{-k}$. Assume that 
$\{ \log (\beta_k/\lambda_k)  / {2^k}: k \in \NN \}$ 
and $\{\lambda_k:k\in\NN\}$ are nondecreasing, positive sequences.  Then there exists a positive constant $c$ such
that for all $f\in \Phi_{\lambda, p} (\TT)$ and $m\in \mathbb{N}$,
we have that
\begin{equation}\label{mk}
\|f-Q_{m}(f)\|_p \ \leq  c \, \|f\|_{\Phi_{\lambda, p}(\TT)} \sum_{k \in \NN} \lambda_{mk}^{-1},
\end{equation}
and consequently,
\begin{equation}  \nonumber
M_{2m+1}(U_{\lambda,p}(\bT),\varphi_\beta)_p
\ \leq \
c\, \sum_{k \in \NN} \lambda_{mk}^{-1}. 
\end{equation}
\end{theorem}

\begin{proof}
Since $\{\lambda_k:k\in\NN\}$ be a nondecreasing sequence, we have that 
$|\Delta \lambda_k^{-1} |= \lambda_k^{-1} -\lambda_{k+1}^{-1}$ for all $k \in \NN$ and then it follows that
\begin{align}\label{ab2}
\sum_{|k| >m} |\Delta \lambda_k^{-1}|\leq  2 \lambda_m^{-1}.
\end{align}
From the inequalities 
$\log \alpha_{k+1} / 2^{k+1} \geq  \log\alpha_k / 2^k \geq  \log\alpha_{k-1} / 2^{k-1}$ we have that
$$
\log \alpha_{k+1} / 2^{k+1} \geq (\log \alpha_{k+1}+ \log \alpha_{k} ) / (2^{k} +  2^{k-1})\geq (\log \alpha_{k+1}+ \log \alpha_{k} ) /2^{k+1}.
$$
Hence, we obtain that 
$$
\log \alpha_{k+1}\geq \log \alpha_{k} + \log \alpha_{k-1}
$$
and also that
$$
\alpha_{k+1}/ \alpha_k \geq \alpha_{k-1} \geq \alpha_{{(k+1)}'}/\alpha_{k'}.
$$
Then, it follows from the hypothesis that 
$\gamma_k= \alpha_{k'}/ (\lambda_k \alpha_k)$ and $\{\lambda_k:k\in\NN\}$ are both nondecreasing, positive sequences, from which we deduce for all $k \in \mathbb{N}$ that 
$$
\gamma_{k+1}\leq \gamma_k\leq \lambda_k^{-1}.
$$
Consequently, we conclude that
\begin{align}\label{ab3}
\sum_{|k| >m} |\Delta \gamma_k| + \sum_{k\in \mathbb{Z}}
|\gamma_{k(2m+1)+m}|=2(\gamma_{m+1} + \sum_{k\in \mathbb{N}}
|\gamma_{k(2m+1)+m}|)\leq 4\sum_{k \in \NN} \lambda_{mk}^{-1}.
\end{align}
From inequalities (\ref{ab2}), (\ref{ab3}) and Theorem \ref{2.1} we confirm (\ref{mk})
which completes the proof the theorem.
\end{proof}

\begin{remark}
Note that for the sequence $\lambda$ defined as
\begin{equation*}
\lambda_j =
\begin{cases}
|j|^r &\text{if }  j \ne 0, \\
1 &\text{if } j=0,
\end{cases}
\end{equation*}
$\Phi_{\lambda, p}$ becomes the Korobov space $K^r_p(\TT)$ 
and from Theorem \ref{theorem[error]} we derive the following estimate 
\begin{equation*}
\|f-Q_{m}(f)\|_p \ \leq  c m^{-r} \|f\|_{K^r_p(\TT)},
\end{equation*}
which have been proven is in \cite{1}.
\end{remark}

 From Theorem \ref{theorem[error]} we immediately derive the following corollary.
\begin{coro}
Let  $1 < p < \infty$, $\lambda_k=\beta_k=e^{-s|k|}$  for all $k\in \ZZ$ where $s>0$. 
Then there exists a positive constant $c$ such that
for all $f\in \Phi_{\lambda, p} (\TT)$ and $m\in \mathbb{N}$, we
have that
\begin{equation} \nonumber
\|f-Q_{m}(f)\|_p \ \leq \ c\, e^{-sm} \|f\|_{\Phi_{\lambda, p} (\TT)},
\end{equation}
and consequently,
\begin{equation}  \nonumber
M_{2m+1}(U_{\lambda,p}(\bT),\varphi_\beta)_p
\ \leq \
c\, e^{-sm}. 
\end{equation}
\end{coro}

\begin{defn}
Let $\beta >0$. A function $b:\ \mathbb{R} \to \mathbb{R}$ will
be called a mask of  type $\beta$ if $b$ is an even function, twice
continuously differentiable such that for $t>0$, $b(t)=
(1+|t|)^{-\beta} F_b (\log |t|)$ for some function $F_b : \RR \to \RR$, where for some constant $c(b)$
$|F_b^{(k)} (t)| \leq c(b)$ for all  $t>1$ and $k=0,1$. A sequence
$\{b_k: k \in \mathbb{Z}\}$ will be called a sequence mask of  type
$\beta$ if there exists a mask $b$ of type $\beta$ such that $b_k = b(k)$ for all $k \in \ZZ$.
\end{defn}

In the next two theorems and their proofs we use the abbreviated notation: 
$\overline{\lambda}_k:=\lambda_k^{-1}$ and $\overline{\beta}_k:=\beta_k^{-1}$.

\begin{theorem} \label{weakype}
Let $1 < p <  \infty$ and the sequence 
$\{ \overline{\lambda}_k:k\in \mathbb{Z}\}=\{\overline{\beta}_k:k\in \mathbb{Z}\}$ be a
sequence mask of  type $r>1$. Then there exists a positive
constant $c$ such that for all  $f\in \Phi_{\lambda, p} (\TT)$ and
$m\in \mathbb{N}$,
\begin{equation*}
\|f-Q_{ m}(f)\|_p \ \ \leq c\, m^{-r} \   \|f\|_{\Phi_{\lambda, p} (\TT)},
\end{equation*}
and consequently,
\begin{equation}  \nonumber
M_{2m+1}(U_{\lambda,p}(\bT),\varphi_\beta)_p
\ \leq \
c\,  m^{-r}. 
\end{equation}
\end{theorem}

\begin{proof}
According the hypothesis we have, by Theorem \ref{theorem[error]}, that
\begin{equation*}
\|f-Q_{ m}(f)\|_p  \leq c \varepsilon_m \   \|f\|_{\Phi_{\lambda, p} (\TT)},
\end{equation*}
where
\[
\varepsilon_m
=
\sum_{|k| >m} |\Delta \lambda_k| + \sum_{k\in \mathbb{Z}}  |\lambda_{k(2m+1)+m}| .
\]
Note that for $k \in \NN$
\[
|\Delta \lambda_k^{-1} | 
= 
| \overline{\lambda}^{'} (k) - \overline{\lambda}^{'} (k+1) | 
= 
| \overline{\lambda}^{'} (c)|,
\]
where $c \in (k, k+1)$.
Therefore, we have for $k \in \NN$, that
\[
\begin{split}
| \overline{\lambda}^{'} (c)| 
&= 
|(1+ c)^{-(r+1)}( \frac{-1}{r} F_{\overline{\lambda}}
(\log c) + F_{\overline{\lambda}}^{'} (\log c) ) |
 \\[1ex]
&\leq 
\frac{r+1}{r}c(\overline{\lambda})
(1+c)^{-(r+1)} \leq   c(\overline{\lambda}) \frac{r+1}{r} (1+k)^{-(r+1)}.
\end{split}
\]
Consequently, we conclude that
\begin{equation}\label{e1}
\sum_{|k| >m} |\Delta \lambda_k^{-1} | \leq 2
c(\overline{\lambda}) \frac{r+1}{r}  \sum_{k >m}(1+k)^{-(r+1)}
\leq c(\overline{\lambda}) \frac{r+1}{r} m^{-r}.
\end{equation}
We also have that
\begin{equation}\label{e2}
 \sum_{k\in \mathbb{Z}}  |\lambda^{-1}_{k(2m+1)+m}| \leq   c(\overline{\lambda}) \sum_{k\in \mathbb{Z}}
(1+ {k(2m+1)+m})^{-r}\leq c(\overline{\lambda}) c_1 m^{-r},
\end{equation}
where $c_1= \sum_{k\in \mathbb{N}} \frac{1}{k^r}.$ We complete the
proof by using equations (\ref{e1}) and (\ref{e2}).
\end{proof}

\begin{defn}
 A function $b: \RR \to \RR$ will be called a function of exponent-type if $b$ is
 two times continuously differentiable and there exists a positive constant
 $s$ such that $b(t)= e^{-s |t| } F_b (|t|)$ for some decreasing function 
$F_b : [0, \infty ) \to [0, \infty ).$
The sequence $\{b_k: k\in \ZZ\}$ will be called a sequence mask of exponent-type if there exists 
a function $b$ of exponent-type such that $b_k = b(k)$ for all $k \in \ZZ$.
\end{defn}

\begin{theorem}
Let $1 < p <  \infty$ and  the sequence   
$\{\overline{\lambda}_k:k \in \ZZ\}$ be a sequence mask of type $r>1$, the sequence
$\{\overline{\beta}_k: k\in \ZZ\}$  be sequence mask of exponent-type. Then there
exists a positive constant $c$ such that for all  $f\in
\Phi_{\lambda, p} (\TT)$ and $m\in \mathbb{N}$,
\begin{equation} \nonumber
\|f-Q_{m}(f)\|_p  \ \leq \ c\, m^{-r} \|f\|_{\Phi_{\lambda, p} (\TT)},
\end{equation}
and consequently,
\begin{equation}  \nonumber
M_{2m+1}(U_{\lambda,p}(\bT),\varphi_\beta)_p
\ \leq \
c\,  m^{-r}. 
\end{equation}
\end{theorem}

\begin{proof}
We have proven in Theorem \ref{weakype} that
$$\sum_{|k| >m} |\Delta \lambda_k|  \leq c m^{-r}.$$
Also, we have that
\[
\begin{split}
|\gamma_k| 
&= 
|\overline{\lambda}_k \frac{\overline{\beta}_{k}}{\overline{\beta}_{k^{'}}} | =|(1+k)^{-r}
F_{ \overline{\lambda} } (\log |k|) e^{-s(k-k^{'})} \frac{F_{\overline{\beta}} (|k|)}{F_{\overline{\beta}}(k^{'})} | 
\\[1ex]
&\leq
|(1+k)^{-r} c{ \overline{\lambda} }  e^{-s(k-k^{'})} |  \leq c k^{-(r+1)},
\end{split}
\]
and so we obtain that
$$
\sum_{|k| >m} |\Delta \lambda_k| + \sum_{k\in \mathbb{Z}}  |\lambda_{k(2m+1)+m}|  \leq c k^{-r}.$$
The proof is complete.
\end{proof}

\section{Multivariate Approximation} \label{Multivariate approximations}
\setcounter{equation}{0}

\subsection{Error estimates for functions in the space $\Phi_{\lambda, 2} (\mathbb{T}^d)$}

\begin{defn}
For ${\bf k} =(k_1,k_2,\ldots,k_d)\in \ZZd$ we define
\begin{equation} \nonumber
|{\bf k}|_p=
\begin{cases}
 (\sum_{j=1}^d |{k}_j|^p )^{1/p} &\text{ if } 1\leq p <\infty\\
\max_{1\leq j \leq d} |{ k}_j| &\text{ if }  p =\infty.
\end{cases}
\end{equation}
\end{defn}

We introduce the trigonometric polynomial  $H_m$  defined at ${\bf x} \in \TTd$  as
\begin{equation}\nonumber
H_m({\bf  x}) 
\ = \ 
\sum_{|{\bf k}|_\infty \le m} \frac{\beta_{\bf  k}}{\lambda_{\bf  k}} \,e^{i {\bf k} {\bf x}}
:= \ 
\sum_{|{\bf k}|_\infty \le m} \alpha_{\bf k} \,e^{i \bf  k \bf x},
\end{equation}
where we set $\alpha_{\bf k}:= \frac{\beta_{\bf  k}}{\lambda_{\bf  k}}$ for $|{\bf k}|_\infty \le m$.

For a function $f \in \Phi_{\lambda, 2} (\TTd)$ represented as $f= \varphi_{\lambda,d} * g$, $g \in L_2(\TTd)$, we define the operator
\begin{equation}\nonumber
Q_{m}(f):= \ \frac{1}{(2m+1)^d}\sum_{{\bf l} \in \ZZdp: \, 0 \le |{\bf l}|_\infty \le 2m}V_m(g)
(\delta_m {\bf l})\varphi_{\beta,d}(\cdot - \delta_m {\bf l}),
\end{equation}
where $\delta_m :=  2\pi/(2m+1)$ and
$
V_m(g):=H_m *g.
$
We put
\[
\varepsilon_m
:= \ \max \{\sup_{|{\bf  k}|_\infty \ge m}|\lambda_{\bf  k}^{-1}|;  
\sqrt{\sum_{|{\bf j}|_\infty > 0}\Gamma_{m,{\bf j}}^2}  \}, 
\]
where  $\gamma_{\bf k}=\beta_{\bf  k}^{-1}\alpha_{{\bf k}^{'}}$ and ${\bf k}^{'}$ is the unique vector in $\ZZd$ such that
$|{\bf k}^{'}|_\infty \le m$ and $\frac{k_j - k'_j}{2m+1} \in \ZZ $  for all ${\bf j} \in \NN_d$, and
$ \Gamma_{m, {\bf  j}} = \max\limits_{|{\bf  k}|_\infty \le m} |\gamma_{{\bf  k}+ (2m+1){\bf j}}|$.

\begin{defn}
The sequence $\{\theta_{\bf k}:{\bf k}\in \ZZd\}$ will be called a
non decreasing-type sequence  if $\theta_{\bf k}\geq c \theta_{\bf
l}$ for all ${\bf k}, {\bf l}\in \ZZd$ satisfying the inequalities
$|k_j| \geq |l_j|$, ${\bf j} \in \NN_d$.
\end{defn}

In a similar way to the proof of Theorems \ref{2.1} and \ref{theorem[error]} we can prove the following two theorems.

\begin{theorem}
There exists a positive constant $c$ such that  for all
 $f\in \Phi_{\lambda, 2} (\TTd)$ and $m\in \mathbb{N}$, we have that
\begin{equation} \nonumber
\|f-Q_{m}(f)\|_2 \ \leq  c  \varepsilon_m \|f\|_{\Phi_{\lambda, 2} (\TTd)},
\end{equation}
and consequently,
\begin{equation}  \nonumber
M_{(2m+1)^d}(U_{\lambda, 2}(\bT^d), \varphi_\beta)_2
\ \leq \
c\,\varepsilon_m. 
\end{equation}
\end{theorem}

\begin{theorem} \label{Error[d>1,p=2]}
Let $|\beta_{{\bf k}}^{-1}| \le c' \, |\lambda_{{\bf k}}^{-1}|$ for all ${\bf k} \in \ZZd$ with a positive constants $c'$ and  $\{|\lambda_{{\bf k}}|:{\bf k} \in\ZZd\}$ be a non decreasing-type sequence. Then there
exists a positive constant $c$ such that  for all
 $m\in \mathbb{N}$ and $f\in \Phi_{\lambda, 2} (\mathbb{T}^d)$ we have that
\begin{equation*}
\|f-Q_{m}(f)\|_2 
\ \leq \
c\,\|f\|_{\Phi_{\lambda, 2} (\mathbb{T}^d)}\sqrt{ \sum_{|{\bf k}|_\infty > 0} \lambda_{m{\bf k}}^{-2}},
\end{equation*}
and consequently,
\begin{equation}  \nonumber
M_{(2m+1)^d}(U_{\lambda, 2}(\bT^d), \varphi_\beta)_2
\ \leq \
c\, \sqrt{ \sum_{|{\bf k}|_\infty > 0} \lambda_{m{\bf k}}^{-2}}. 
\end{equation}
\end{theorem}

From the above corollary we have the following result.

\begin{coro} \label{[|f-Q_{m}(f)|_2<]}
Let $\Big\{\frac{| \beta_{\bf k} | }{ |{\bf k}|^r |\lambda_{\bf k}|}:{\bf k}\in \ZZd \Big\}$ 
be a non decreasing-type sequence for some $r> \frac{d}{2}$. Then there exists a positive constant $c$ such
that  for all   $m\in \mathbb{N}$ and $f\in \Phi_{\lambda, 2}
(\mathbb{T}^d)$ we have that
\begin{equation*}
\|f-Q_{m}(f)\|_2 
\ \leq \ 
c\, \sup_{|{\bf k}|_\infty > m} |\lambda_{\bf k}^{-1}| \|f\|_{\Phi_{\lambda, 2} (\mathbb{T}^d)},
\end{equation*}
and consequently,
\begin{equation}  \nonumber
M_{(2m+1)^d}(U_{\lambda, 2}(\bT^d), \varphi_\beta)_2
\ \leq \
c\, \sup_{|{\bf k}|_\infty > m} |\lambda_{\bf k}^{-1}|. 
\end{equation}
\end{coro}

\begin{proof}
We see from the hypothesis that for all ${\bf k}$ with $|{\bf k}|_\infty \le m$, we have that
\[
\frac{{|\beta_{{\bf k}+(2m+1){\bf j}}| }}{|\lambda_{{\bf k}+(2m+1){\bf j}}| }
\geq 
c\frac{ |{{\bf k}+(2m+1){\bf j}}|}{ |{\bf k}|} \frac{| \beta_{\bf k}| }{ {| \lambda_{\bf k}} |},
\]
and then it follows that
\[
|\gamma_{{\bf k}+(2m+1){\bf j}}| = |\lambda_{{\bf k}+(2m+1){\bf j}}^{-1}|
\frac{|\lambda_{{\bf k}+(2m+1){\bf j}}|}{{|\beta_{{\bf k}+(2m+1){\bf j}}}|}
\frac{| \beta_{\bf k}|}{ {| \lambda_{\bf k}}| }  
\leq 
c_1
\sup_{|{\bf k}|_\infty > m} |\lambda_{\bf k}^{-1}| |{\bf j}|_1^r.
\]
Hence, we conclude that
\[
\Gamma_{m,{\bf j}} \leq c_1\sup_{|{\bf k}|_\infty > m}|\lambda_{\bf k}^{-1}| |{\bf j}|_1^r 
\]
and so
\[
\sqrt{ \sum_{|{\bf j}|_\infty > 0} \gamma_{m,{\bf j}}^{-2}}  
\leq 
c_1 \sup_{|{\bf k}|_\infty > m}|\lambda_{\bf k}^{-1}| \sum_{|{\bf j}|_\infty > 0}|{\bf j}|_1^{-2r}
\leq
 c_1 \sup_{|{\bf k}|_\infty > m}|\lambda_{\bf k}^{-1}| \sum_{{ j}=1 }^{\infty} { j}^{d-1}  j^{-2r}.
\]
Note that for $2r> d$ the series $\sum_{{ j}=1 }^{\infty} { j}^{d-1} { j}^{-2r}$ is convergent which
completes the proof of the corollary.
\end{proof}

\begin{coro}
Let
$\beta_{\bf k}=\lambda_{\bf k}^2$ for all ${\bf k}\in \ZZd$, and the
sequence $\Big\{\frac{| \lambda_{\bf k} | }{ |{\bf k}|^r }:{\bf k}\in \ZZd \Big\}$ 
be non decreasing-type for some $r> \frac{d}{2}$. Then there
exists a positive constant $c$ such that  for all   
$m\in \mathbb{N}$ and $f\in \Phi_{\lambda, 2} (\mathbb{T}^d)$ we have that
\begin{equation*}
\|f-Q_{m}(f)\|_2 \leq  c \sup_{|{\bf k}|_\infty > m} |\lambda_{\bf k}^{-1}| \|f\|_{\Phi_{\lambda, 2} (\mathbb{T}^d)},
\end{equation*}
and consequently,
\begin{equation}  \nonumber
M_{(2m+1)^d}(U_{\lambda, 2}(\bT^d), \varphi_\beta)_2
\ \leq \
c\,\sup_{|{\bf k}|_\infty > m} |\lambda_{\bf k}^{-1}|. 
\end{equation}
\end{coro}

\begin{coro}
Let
$\beta_{\bf k}=\alpha_{\bf k}$ for all ${\bf k}\in \ZZd$; the
sequence $\Big\{\frac{| \lambda_{\bf k} | }{ |{\bf k}|^r }:{\bf k}\in \ZZd \Big\}$ 
is nondecreasing-type for some $r> \frac{d}{2}$. Then there
exists a positive constant $c$ such that  for all   $m\in \mathbb{N}$ and 
$f\in \Phi_{\lambda, 2} (\mathbb{T}^d)$ we have that
\begin{equation*}
\|f-Q_{m}(f)\|_2 
\ \leq \
 c\, \sup_{|{\bf k}|_\infty > m}
 |\lambda_{\bf k}^{-1}| \|f\|_{\Phi_{\lambda, 2} (\mathbb{T}^d)},
\end{equation*}
and consequently,
\begin{equation}  \nonumber
M_{(2m+1)^d}(U_{\lambda, 2}(\bT^d), \varphi_\beta)_2
\ \leq \
c\, \sup_{|{\bf k}|_\infty > m} |\lambda_{\bf k}^{-1}|. 
\end{equation}
\end{coro}

\subsection{Convergence rate}

\begin{theorem}
Let  $\Psi: \ \mathbb{R}_+ \to \mathbb{R}_+$ be a nondecreasing function such that
$\Psi (2t) \leq c \Psi (t)$ for all $t>1$.
If  $\lambda_{\bf k}= \Psi (|{\bf k}|_2)$ for all ${\bf k} \in \ZZd$ then there exist positive constants 
$c'$ and $c''$ such that for all $n \in \mathbb{N}$,
$$c' / \Psi ( (n \log n)^{1/d}) \leq  M_n (U_{\lambda, 2})_2 \leq  c''/  \Psi (n^{1/d}) .$$
\end{theorem}
\begin{proof}
Let $s$ be any natural number and set
$$
s^*= |\{ {\bf k}: {\bf k}\in \ZZd, \quad  |{\bf k}|_2 \leq s
\}|$$
where
$|A|$ denoted by the number of elements of a finite set $A$.  
Then, for some positive constants $c_1$ and $c_2$, we have that
\begin{equation}\label{ss}
c_1 s^d \leq s^* \leq c_2 s^d.
\end{equation}
Let $n$ be any natural number satisfying $n\geq 10$. We define the positive integers
$$
m= \lfloor c_3 n \log n \rfloor + 1$$
and $s=s(n)$ as the largest natural number satisfying the inequality 
$$
m \geq s^*,
$$
where $c_3 = \sqrt{4ec_1}$. Then it follows from (\ref{ss}) that there exists a positive number $c_4$ independent of $n$ such that
\begin{equation}\label{ssss}
s \leq c_4 (n \log n)^{1/d}.
\end{equation}
We consider the set of trigonometric polynomials
$$
F_{n,s}=\biggl\{\omega \sum_{|{\bf k}|_2 \leq s} \epsilon_{\bf k} e^{i{\bf k} {\bf x}}: 
\quad |\epsilon_{\bf k}|=1, {|{\bf k}|_2 \leq s} \biggl\},
$$
where
$$
\omega=m^{-1/2} \Big( \max_{|{\bf k}|_2=s}|\lambda_{\bf k} | \Big)^{-1}.
$$
Let 
$$
f(x)=\sum_{|{\bf k}|_2\leq s}\hat{f}_{\bf k} e^{i{\bf k} {\bf x}}
$$ 
be any polynomial from $F_{n,s}$.
Since
\[
\| f\|_{\Phi_\lambda}
=  
\omega \Big( \sum_{|{\bf k}|_2 \leq s} \lambda_{\bf k}^2 \Big)^{1/2}
\leq
\omega \max_{|{\bf k}|_2=s}|\lambda_{\bf k}| \Big( \sum_{|{\bf k}|_2 \leq s} 1 \Big)^{1/2} =
\omega \max_{|{\bf k}|_2=s}|\lambda_{\bf k}| \Big( s^* \Big)^{1/2} \leq 1,
\]
$f$ belongs to $U_{\lambda,2}$ and consequently, $F_{n,s} \subset U_{\lambda,2}$.
Take an arbitrary function $\psi$ from $L_2(\TTd)$.
Then  it follows from a result in \cite{Ma05} that there exists a function $f^* \in F_{n,s}$ and a positive number $c_5$ such that
for any  linear combination of translates of $\psi$
\[
h(x) =\sum_{l=1}^n
b_l \psi (x-a_l),
\]
we have that
\[
\|f^* - h\|_2\geq c_5 \omega m^{1/2}\geq c_5 \Big( \max_{|{\bf k}|_2\leq s}|\lambda_{\bf k} |\Big)^{-1}.
\]
Therefore, by using (\ref{ssss}) and $\Psi (2t) \leq c \Psi (t)$ for all $t>1$,  there exists a positive constant $c'$ independent of $n$  such that
\[
\begin{split}
M_n (U_{\lambda, 2})_2 
&\geq 
M_n (F_{n,s})_2 
\geq \|f^* - h\|_2
\\[1ex] 
&\geq  
c_5 \Big( \max_{|{\bf k}|_2\leq s}|\lambda_{\bf k} |\Big)^{-1}
\geq c_5 (\Psi (s))^{-1}
\geq
c' / \Psi ( (n \log n)^{1/d})
\end{split} 
\]
which proves the lower bound of the theorem.

By using Corollary \ref{[|f-Q_{m}(f)|_2<]} for $\beta_{\bf k}= |\bf k|^d |\lambda_{\bf k}|$ and $m= (2u+1)^d<n$,  there exists a positive constant $c_6$ independent of $n$  such that 
$$
M_n (U_{\lambda, 2})_2 \leq M_m (U_{\lambda, 2})_2 \leq M_m (U_{\lambda, 2}, \psi)_2\leq c_6 \sup_{|{\bf k }|_\infty > u} |\lambda_{\bf k}^{-1}|,
$$
where $\psi (x)=\sum_{\bf k \in \mathbb{Z}^d} e^{i{\bf k}{\bf x}} \lambda_{\bf k}$ and 
$u= \lfloor \sqrt[d]{n}/2 \rfloor - 1.$
That gives us the inequality
\begin{equation} \nonumber
M_n (U_{\lambda, 2})_2 \leq c_6\sup_{{\bf k}\in \ZZd \setminus [-u,u]^d}|\lambda_{\bf k}^{-1}|
\end{equation}
and then we get that
$$M_n (U_{\lambda, 2})_2 \leq c_6 (\Psi (u))^{-1}.$$
Hence, by using $\Psi (2t) \leq c \Psi (t)$ for all $t>1$, we obtain the remaining desired result
$$M_n (U_{\lambda, 2})_2\leq c''/  \Psi (n^{1/d}).$$
The proof is complete.
\end{proof}

\bigskip
{\bf Acknowledgments}
Dinh D\~ung's research is funded by Vietnam National
Foundation for Science and Technology Development
(NAFOSTED) under  Grant No. 102.01-2017.05.
 Dinh D\~ung and Vu Nhat Huy thank
 Vietnam Institute for Advanced
Study in Mathematics (VIASM) for providing a
fruitful research environment and working condition.
In addition, Charles Micchelli wishes to acknowledge partial support from NSF under Grant DMS 1522339.

\end{document}